\documentclass [a4paper,12pt]{amsart}
\usepackage{graphicx}
\usepackage[ansinew]{inputenc}    
\usepackage[all]{xy}

\usepackage{amssymb,amscd}

\newtheorem{thm}{Theorem}[section]

\def\C{\mathbb C}

\def\dim{\operatorname{dim}}

\def\rank{\operatorname{rank}}

\usepackage{graphicx}
\usepackage[ansinew]{inputenc}    
\usepackage[all]{xy}
\usepackage{amssymb,amscd}
\usepackage{color}

\newtheorem{cor}[thm]{Corollary}
\newtheorem{teo}[thm]{Theorem}
\newtheorem{lem}[thm]{Lemma}

\theoremstyle{definition}

\newtheorem{defi}[thm]{Definition}

\def\C{\mathbb C}

\def\dim{\operatorname{dim}}

\def\rank{\operatorname{rank}}

\begin{document}
\title[Equisingularity of families of functions on IDS]{Equisingularity of families of functions on isolated determinantal singularities}

\author{R. S. Carvalho, J. J. Nuño-Ballesteros, B. Oréfice-Okamoto, J. N. Tomazella}

\address{Departamento de Matem\'atica, Universidade Federal de S\~ao Carlos, Caixa Postal 676,
13560-905, S\~ao Carlos, SP, BRAZIL}

\email{rafaelasoares@dm.ufscar.br}

\address{Departament de Matemàtiques, Universitat de València, Campus de Burjassot, 46100 Burjassot SPAIN}

\email{Juan.Nuno@uv.es}

\address{Departamento de Matem\'atica, Universidade Federal de S\~ao Carlos, Caixa Postal 676,
13560-905, S\~ao Carlos, SP, BRAZIL}

\email{bruna@dm.ufscar.br}

\address{Departamento de Matem\'atica, Universidade Federal de S\~ao Carlos, Caixa Postal 676,
13560-905, S\~ao Carlos, SP, BRAZIL}

\email{tomazella@dm.ufscar.br}

\thanks{The first author was partially supported by CAPES.
The second author was partially supported by by MICINN Grant PGC2018--094889--B--I00 and by GVA Grant AICO/2019/024. The third author was partially supported by FAPESP Grant
2016/25730-0. The fourth author was partially supported by CNPq Grant
309086/2017-5 and FAPESP Grant 2018/22090-5.}


\begin{abstract} We study the equisingularity of a family of function germs $\{f_t\colon(X_t,0)\to (\mathbb{C},0)\}$, where $(X_t,0)$ are $d$-dimensional isolated determinantal singularities. We define the $(d-1)$th polar multiplicity of the fibers $X_t\cap f_t^{-1}(0)$ and we show how the constancy of the polar multiplicities is related to the constancy of the Milnor number of $f_t$ and the Whitney equisingularity of the family.

\end{abstract}

\maketitle
\section{Introduction}

We consider analytic families of function germs $\{f_t\colon(X_t,0)\to (\mathbb{C},0)\}$ with isolated critical points, where $(X_t,0)$ are $d$-dimensional isolated determinantal singularities (IDS). To be more precise, we assume that $(\mathcal X,0)$ is a variety in $(\C\times \C^N,0)$ of dimension $d+1$, given by zero locus of the ideal $I_s(A)$ generated by the minors of size $s$ of a matrix $A=(a_{ij})$ of size $m\times n$, with $a_{ij}\in\mathcal O_{N+1}$ and such that $d=N-(m-s+1)(n-s+1)$. For each $t\in\C$ close to $0$, $X_t=\pi^{-1}(t)$, where $\pi\colon(\mathcal X,0)\to(\C,0)$ is the projection $(t,x)\mapsto t$. 
The family of functions $f_t$ is constructed by taking an unfolding, that is, a holomorphic map $F\colon(\mathcal X,0)\to(\C\times\C,0)$ of the form $F(t,x)=(t,f_t(x))$. 
We assume that $0$ is an isolated critical point of $f_t\colon X_t\to\C$, for all $t$ (a critical point is either a singular point $x\in X_t$ or a regular point $x\in X_t$ which is critical point of $f_t$ in the usual sense).

Our goal is to characterize the Whitney equisingularity of the family of functions by means of analytic invariants of each member of the family. The first step is to assume that the family is good. That is, there exists an open neighbourhood of the origin $D\times U$ in $\C\times\C^{N}$ such that $0$ is the only critical point of $f_t\colon X_t\to\C$ on $U$. 

If the family is good, we have a natural stratification $(\mathcal A,\mathcal A')$ of $F\colon\mathcal X\to D\times\C$ given by
\[
\mathcal A=\{\mathcal{X}\setminus F^{-1}(T),F^{-1}(T)\setminus S,S\},\quad \mathcal A'=\{(D \times \C)\setminus T,T\},
\] 
where $S=D\times \{0\}\subset \mathbb{C}\times\mathbb{C}^N$ and $T=D\times\{0\}\subset \mathbb{C}\times\mathbb{C}$. We say that the family is Whitney equisingular if this is a regular stratification of $F$ (that is, $\mathcal A$ and $\mathcal A'$ are Whitney stratifications and $F$ satisfies the Thom condition).

In a previous paper \cite{Bruna 3}, some of the authors showed that the family of varieties $\{(X_t,0)\}$ is Whitney equisingular if and only if it is good and all the polar multiplicities $m_i(X_t,0)$, $i=0,\dots,d$ are constant on $t$.  For $i=0,\dots,d-1$, $m_i(X_t,0)$ are the usual polar multiplicities of any $d$-dimensional variety (see for instance \cite{Le Teissier}). The top polar multiplicity $m_d(X_t,0)$ was introduced in \cite{Bruna 2} for IDS. It is equal to the number of critical points of a generic linear form on a determinantal smoothing $X_{t,s}$ of $X_t$. 

In our case, we also need to control the Whitney conditions of the new stratum $F^{-1}(T)\setminus S$  with respect to $S$. It seems natural that this could be done by looking at the polar multiplicities of the family of fibers $\{(Y_t,0)\}$, where $Y_t=X_t\cap f_t^{-1}(0)$. The problem is that $Y_t$ is not a determinantal singularity in the usual sense, but some kind of ``nested'' determinantal singularity, i.e., a determinantal singularity in an ambient space which is also determinantal (of a different type).

We introduce the top polar multiplicity of the fiber $m_{d-1}(Y_t,0)$ as the number of critical points of a generic linear form on a convenient smoothing $Y_{t,s}$ of $Y_t$. It has the property that the Euler characteristic of the smoothing is
\[
\chi(Y_{t,s})=\sum_{i=0}^{d-1} (-1)^i m_i(Y_t,0).
\]
This formula can be  rewritten as a relation between $m_{d-1}(Y_t,0)$, the Euler obstruction $Eu(Y_t,0)$ and the vanishing Euler characteristic $\nu(Y_t,0)=(-1)^{d-1}(\chi(Y_{t,s})-1)$:
 $$Eu(Y_t,0)+(-1)^{d-1}m_{d-1}(Y_t,0)=1+(-1)^{d-1}\nu(Y_t,0),$$
 (see Corollaries \ref{soma alternada das multiplicidades} and \ref{formula para a obstrucao de euler da fibra}).

The main result of the paper (Theorem \ref{main}) is that the family $\{f_t\colon(X_t,0)\to (\mathbb{C},0)\}$ is Whitney equisingular if and only if it is good and all the polar multiplicities $m_i(X_t,0)$, $i=0,\dots,d$ and $m_i(Y_t,0)$, $i=0,\dots,d-1$ are constant on $t$. Since we have Lê-Greuel type formulas for the vanishing Euler characteristics of $(X_t,0)$ and $(Y_t,0)$, the constancy of the polar multiplicities $m_i(X_t,0)$, $i=0,\dots,d$ and $m_i(Y_t,0)$, $i=0,\dots,d-1$  is equivalent to the constancy of the $\nu^*$-sequences of $(X_t,0)$ and $(Y_t,0)$ (obtained by the taking sections by generic planes of  codimensions $0,\dots,d$ and $0,\dots,d-1$ respectively). 

We include an appendix at the end to show that given a holomorphic function $g\colon Y\to \C$ on a variety $Y$, the set of points $y\in Y$ such that either $y$ is a singular point of $Y$ or $y$ is a degenerate critical point of $g$ is analytic.

\section{Isolated determinantal singularities}

Throughout this paper, if $A$ is a matrix we denote by $I_r(A)$ the ideal generated by the minors of size $r$ of $A$. If $H\colon(\mathbb{C}^N,0)\rightarrow(\mathbb{C}^p,0)$ is a holomorphic function germ we denote by $J(H)$ the Jacobian matrix of $H$.

Let $M_{m,n}$ be the set of complex matrices of size $m\times n$ and $M_{m,n}^s$ the subset of $M_{m,n}$ of the matrices with $\rank$ less than $s$, where $0<s\leq m\leq n$ are integer numbers. We remark that $M_{m,n}^s$ is an algebraic variety of $M_{m,n}$ of codimension $(m-s+1)(n-s+1)$ (see \cite{Arbarello}). 

Let $\psi\colon(\mathbb{C}^N,0)\rightarrow
(M_{m,n},0)$ be a holomorphic map germ and $(X,0)\subset (\mathbb{C}^N,0)$ the germ $\psi^{-1}(M_{m,n}^s)$. We say that $(X,0)$ is a determinantal singularity of type $(m,n;s)$ if it has the expected codimension, that is, $$\dim(X,0)=N-(m-s+1)(n-s+1).$$ Moreover, if $s=1$ or
$N<(m-s+2)(n-s+2)$ and $X$ is smooth at $x\in X$ for $x\neq 0$ in a small neighborhood of $0$, we say that $(X,0)$ is an
isolated determinantal singularity, shortening (IDS) (see \cite{Bruna 2}).

We consider $X=\psi^{-1}(M_{m,n}^s)$ a small enough representative of $(X,0)$, where $\psi\colon B \to M_{m,n}$ is defined on a small enough open ball $B=B_{\epsilon}$ centered at the origin in $\mathbb{C}^N$. As a generalization for the Milnor number of an ICIS, the vanishing Euler characteristic of the IDS $(X,0)$ is defined in \cite{Bruna 2} by $$\nu(X,0)=(-1)^{\dim X}(\chi(X_A)-1),$$ where $A\in M_{m,n}$ is a generic matrix and $X_A=\psi_A^{-1}(M_{m,n}^s)$, with $\psi_A\colon B\to M_{m,n}$ defined by $\psi_A(x)=\psi(x)+A$.

A map germ $\Psi\colon(\mathbb{C}\times\mathbb{C}^N,0)\rightarrow
M_{m,n}$ such that $\Psi(0,x)=\psi(x)$ for all $x\in\mathbb{C}^N$ is
called determinantal deformation of $(X,0)$. We set
$\Psi(t,x):=\psi_t(x)$ and $X_t:={\psi_t}^{-1}(M_{m,n}^s)$. We say that a determinantal deformation is a determinantal smoothing if $X_t$ is smooth for $t\neq 0$ sufficiently small. In this case, $\nu(X,0)=(-1)^{\dim X}(\chi(X_t)-1)$ (see \cite[Theorem
3.4]{Bruna 2}).

For $f\colon(X,0)\rightarrow \mathbb{C}$ a holomorphic function germ, we choose $f\colon X\to \mathbb{C}$ a representative of $f$, defined in the representative of $(X,0)$ above-mentioned. In \cite{Bruna 2} the Milnor number of $f$ is defined by $$\mu(f)=\sharp S(f_b|_{X_A}),$$ where $f_b\colon X_A\to \mathbb{C}$ is defined by $f_b(x_1,\ldots,x_N)=f(x_1,\ldots,x_N)+b_1x_1+\ldots+b_Nx_N$ for $b=(b_1,\ldots,b_N)\in \mathbb{C}^N$ generic and $\sharp S(f_b|_{X_A})$ is the number of critical points of $f_b|_{X_A}$.

Moreover, in \cite{Bruna 2}, it is defined the vanishing Euler characteristic of the fiber $X\cap f^{-1}(0)$ by $$\nu(X\cap f^{-1}(0),0):=(-1)^{\dim X-1}(\chi(X_A\cap f_b^{-1}(e))-1),$$ with $(b,A,e)\in\mathbb{C}^N\times M_{m,n}\times
\mathbb{C}$ generic values such that $X_A$ is smooth,
$f_b|_{X_A}$ is a Morse function and $e$ is a regular value of
$f_b|_{X_A}$.

Also, in \cite{Bruna 2}, it is proved that 
\begin{equation}\label{Le-Greuel}
\mu(f)=\nu(X,0)+\nu(X\cap f^{-1}(0),0),
\end{equation} 
that is, these invariants satisfy the well-known Lê-Greuel's formula.

For $(X,0)\subset(\mathbb{C}^N,0)$ any $d$-dimensional variety, Lê and Teissier in \cite{Le Teissier}, considering a generic linear projection $p\colon\mathbb{C}^N\to \mathbb{C}^{d-k+1}$ with respect to $X$, define the $k$th polar multiplicity of $(X,0)$, for $k=0,\ldots,d-1$, by $$m_k(X,0)=m_0(\overline{S(p|_{X^0})},0),$$ where $X^0$ denotes the smooth part of $X$, $S(p|_{X^0})$ is the set of critical points of $p|_{X^0}$ and $m_0(Z,0)$ is the usual multiplicity of any variety $(Z,0)$.

In the case of an ICIS $(X,0)=((\phi_1,\ldots,\phi_{N-d})^{-1}(0),0)\subset(\mathbb{C}^N,0)$, Gaffney in \cite{Gaffney definicao de multiplicidade} defines the $d$th polar multiplicity of $(X,0)$ by

\begin{center} $m_d(X,0):=\displaystyle\dim_{\mathbb{C}}\frac{\mathcal{O}_N}{\langle \phi_1,\ldots,\phi_{N-d}\rangle+I_{N-d+1}(J_x(\phi_1,\ldots,\phi_{N-d},p))}$,  \end{center} where $p\colon\mathbb{C}^N\to \mathbb{C}$ is a generic linear function. Thus, $m_d(X,0)=\sharp S(p|_{X_t})$, where $X_t$ is the Milnor fiber of $(X,0)$.

It is defined in \cite{Bruna 2} the $d$th polar multiplicity of an IDS of dimension $d$ by $$m_d(X,0)=\sharp S(p|_{X_A}),$$ where $p\colon\mathbb{C}^N\to \mathbb{C}$ is a generic linear projection.

\section{Equisingularity of families of function germs}

Let $(X,0)$ be the IDS defined by a holomorphic map germ $\psi\colon(\mathbb{C}^N,0)\rightarrow M_{m,n}$ and $f\colon(X,0)\rightarrow(\mathbb{C},0)$ a function germ with isolated singularity, that is, 
in a sufficiently small neighborhood of $0$, $f$ is regular on $X \setminus\{0\}$.

Let $\Psi\colon(\mathbb{C}\times \mathbb{C}^N,0)\rightarrow M_{m,n}$ be a determinantal deformation of $(X,0)$ and $$\begin{array}{cccc}
F\colon & (\mathcal{X},0) & \to & (\mathbb{C}\times \mathbb{C},0) \\
& (t,x) & \mapsto & F(t,x):=(t,f_t(x))
\end{array}$$ an unfolding of $f$, where $\mathcal{X}:=\Psi^{-1}(M_{m,n}^s)$. We assume that $F$ is origin preserving, that is, $0\in X_t$ and $f_t(0)=0$ for $t$ small enough. So, we can see the unfolding $F$ as a 1-parameter family of map germs $\{f_t\colon(X_t,0)\to (\C,0)\}_{t\in D}$, where $D$ is an open neighborhood of the origin in $\mathbb{C}$.

\begin{defi} We say:

{\rm(1)} $(\mathcal{X},0)$ is $\nu$-constant if $\nu(X_t,0)=\nu(X,0)$ for $t$ small enough;

{\rm(2)} $(\mathcal{X},0)$ is a good family if there exists a representative defined in $D\times U$, where $D$ and $U$ are open neighbourhoods of the origin in $\mathbb{C}$ and $\mathbb{C}^N$ respectively, such that $X_t\setminus \{0\}$ is smooth, for any $t\in D$, that is, $S(X_t)=\{0\}$ on $U$, for all $t\in D$, where $S(X_t)$ is the singular set of $X_t$;

{\rm(3)} $(\mathcal{X},0)$ is topologically trivial if there is a homeomorphism $H\colon(\mathcal{X},0)\rightarrow(\mathbb{C}\times X,0)$ which commutes with the projection, that is, $\pi\circ H=\pi$, where $\pi\colon(\mathcal{X},0)\rightarrow (\mathbb{C},0)$ is given by $\pi(t,x)=t$;

{\rm(4)} $(\mathcal{X},0)$ is Whitney equisingular if it is a good family and there exists a representative as in item (2) such that $(\mathcal{X}\setminus T,T)$ satisfies Whitney conditions, where $T=D\times\{0\}$; 

{\rm(5)} $F$ is $\mu$-constant if $\mu(f_t)=\mu(f)$ for $t$ small enough;

{\rm(6)} $F$ is good if there is a representative defined in $D\times U$, where $D$ and $U$ are open neighbourhoods of the origin in $\mathbb{C}$ and $\mathbb{C}^N$ respectively, such that $X_t\setminus \{0\}$ is smooth and $f_t$ is regular on $X_t\setminus \{0\}$, for any $t\in D$;

\rm{(7)} $F$ is topologically trivial if there are homeomorphism map germs $$\begin{array}{cccc}
G\colon & (\mathcal{X},0) & \to & (\mathbb{C}\times X,0) \\
& (t,x) & \mapsto & G(t,x)=(t,{g}_t(x))
\end{array}$$ and $$\begin{array}{cccc}
H\colon & (\mathbb{C}\times\mathbb{C},0) & \to & (\mathbb{C}\times \mathbb{C},0) \\
& (t,y) & \mapsto & H(t,y)=(t,{h}_t(y))
\end{array}$$ such that $G$ and $H$ are unfoldings of the identity and $F=H\circ U\circ G$, where $U(t,x)=(t,f(x))$ is the trivial unfolding of $f$;

{\rm(8)} $F$ is Whitney equisingular if it is a good family and there is a representative as in item (6) which admits a regular stratification given by $\mathcal A=\{\mathcal{X}\setminus F^{-1}(T),F^{-1}(T)\setminus S,S\}$ in the source  and $\mathcal A'=\{(\C\times \C)\setminus T,T\}$ in the target, where 
$S=D\times \{0\}\subset \mathbb{C}\times\mathbb{C}^N$ and $T=D\times\{0\}\subset \mathbb{C}\times\mathbb{C}$.
\end{defi}

When we consider families of singularities, it is interesting to know which is the relationship between the topological triviality and the Whitney equisingularity of the family. Moreover, we also want invariants, whose constancy in the family characterizes the topological triviality or the Whitney equisingularity.

For instance, for families of space curves $(\mathcal{X},0)=\{(X_t,0)\}$, we can see in \cite{Briancon1} and \cite{Buchweitz-Greuel}:

\begin{enumerate}
\item If $(\mathcal{X},0)$ is $\mu$-constant, then it is good;

\item $(\mathcal{X},0)$ is topologically trivial and good if and only if it is $\mu$-constant;

\item $(\mathcal{X},0)$ is Whitney equisingular if and only if it is $\mu$-constant and $m_0(X_t,0)$ is constant, where $m_0(X_t,0)$ is the multiplicity of $(X_t,0)$.
\end{enumerate}

Analogously, for a family $\{f_t\colon(X_t,0)\to(\C,0)\}$ of functions on curves we can find in \cite{tomazela e juan} the following:

\begin{enumerate}
\item If $F$ is $\mu$-constant, then it is good;

\item $F$ is topologically trivial and good if and only if it is $\mu$-constant;

\item $F$ is Whitney equisingular if and only if it is $\mu$-constant and $m_0(X_t,0)$ is constant.
\end{enumerate}

For a good family of IDS $(\mathcal{X},0)=\{(X_t,0)\}$ of dimension $d$, we see in \cite{Bruna 3} that 
$(\mathcal{X},0)$ is Whitney equisingular if and only if the polar multiplicities $m_i(X_t,0)$, $i=0,\ldots,d$, are all constant on $t$. This extends the same result for families of hypersurfaces by Teissier \cite{Teissier} and later for families of ICIS by Gaffney \cite{Gaffney 1}. Moreover, in the case of families of ICIS we also have that a $\mu$-constant family is always good, which is not true in general for families of IDS. If $d\ne2$, the $\mu$-constant condition also controls the topological triviality of a family of hypersurfaces (see \cite{Le Ramanujam}) or ICIS (see \cite{Parameswaran}). This result has been also extended recently for families of IDS in \cite{Bruna 3}.

Our goal here is to characterize the Whitney equisingularity by means of the constancy of some invariants in a family of function germs on IDS $$F\colon(\mathcal{X},0)\rightarrow (\mathbb{C}\times \mathbb{C},0).$$

Throughout this section, let $f\colon(X,0)\to (\C,0)$ be a function on an IDS with isolated singularity and let $F\colon(\mathcal X,0)\to(\C\times\C,0)$ be an origin preserving unfolding of $f$.

We can see in \cite[Lemma 
5.1]{Bruna 3} that if $p\colon\mathbb{C}^N\to \mathbb{C}$ is a generic linear projection then we have the conservation of the Milnor number of $p$, that is,
$$
\mu(p|_{X})=\displaystyle\sum_{y\in S(p|_{X_t^0})}\mu(p|_{X_t},y)+\sum_{x\in S({X_t})}\mu(p|_{X_t},x)=\displaystyle\sum_{x\in S(p|_{X_t})}\mu(p|_{X_t},x),
$$
for all $t$ small enough. In fact, such formula is presented in \cite[Lemma 
5.1]{Bruna 3} as follows, 
$$
\mu(p|_{X})=\displaystyle\sum_{y\in S(p|_{X_t^0})}\mu(p|_{X_t},y)+\sum_{x\in S({X_t})}m_d(X_t,x),$$  where $X_t^0$ is the smooth part of $X_t$ and $S(X_t)$ is the singular locus of $X_t$.

We prove in the following theorem that the same conservation of Milnor number holds for a family $f_t$ of functions.

\begin{teo}\label{formula do numero de milnor} For all $t$ small enough,
$$\mu(f)=\displaystyle\sum_{y\in S(f_t|_{X_t^0})}\mu(f_t,y)+\sum_{x\in S({X_t})}\mu(f_t,x)=\sum_{z\in S({f_t})}\mu(f_t,z).$$
\end{teo}

\begin{proof} We choose $f\colon X\rightarrow \mathbb{C}$ a representative of $f$, where $X=\psi^{-1}(M_{m,n}^s)$ is a small enough representative of $(X,0)$, with $\psi\colon B \to M_{m,n}$ defined on a small enough open ball $B=B_{\epsilon}$ centered at the origin in $\mathbb{C}^N$ such that in $B_{\epsilon}$ the origin is the only singular point of $f$. 
	
	For a matrix $A\in M_{m,n}$, let $\psi_A\colon(\mathbb{C}^N,0)\rightarrow M_{m,n}$ be defined by $$\psi_A(x)=\psi(x)+A$$ and we write $X_A=\psi_A^{-1}(M_{m,n}^s)$.
	
	Given $b=(b_1,\ldots,b_N)\in \mathbb{C}^N$ let $f_b|_{X_A}\colon X_A\to \mathbb{C}$ be defined by $$f_b(x_1,\ldots,x_N)=f(x_1,\ldots,x_N)+b_1x_1+\ldots+ b_Nx_N.$$

	We construct a new deformation as the sum of $X_t$ and $X_A$, that is, for $A\in M_{m,n}$ and $t\in D$ we define $$X_{(A,t)}=(\psi_t+A)^{-1}(M_{m,n}^s).$$
	
	Moreover, we put $$f_{(b,t)}|_{X_{(A,t)}}\colon X_{(A,t)}\to \mathbb{C}, \, \,  f_{(b,t)}(x_1,\ldots,x_N):=f_t(x_1,\ldots,x_N)+b_1x_1+\ldots+b_Nx_N.$$

We denote by $z_1,\ldots,z_k$ and $y_1,\ldots,y_l$ the singular points of $X_t$ and $f_t|_{X_t^0}$, respectively.

For each $i=1,\ldots,k$ and $j=1,\ldots,l$ we choose a Milnor disc $D_i$ for $(X_t,z_i)$ and $E_j$ around $y_j$ for $f_t|_{X_t^0}$ such that these discs are pairwise disjoint. We have \begin{eqnarray*}
	\mu(f)&=&\sharp S(f_{(b,t)}|_{X_{(A,t)}})\\
	&=&\displaystyle\sum_{j=1}^l\sharp S(f_{(b,t)}|_{X_{(A,t)}\cap E_j})+\sum_{i=1}^k\sharp S(f_{(b,t)}|_{X_{(A,t)}\cap D_i})\\
	&=&\displaystyle\sum_{j=1}^l\mu(f_t,y_j)+\sum_{i=1}^k\mu(f_t,z_i).
\end{eqnarray*}  \end{proof}

\begin{cor} \label{teo final} Assume that $(\mathcal{X},0)$ is a  good family. Then $F$ is $\mu$-constant if and only if it is good.
\end{cor}

\begin{proof}

Assume that $F$ is $\mu$-constant. Since $(\mathcal{X},0)$ is good, by Theorem \ref{formula do numero de milnor} $$\mu(f)=\displaystyle\sum_{y\in S(f_t|_{X_t^0})}\mu(f_t,y)+\mu(f_t).$$

Then $\displaystyle\sum_{y\in S(f_t|_{X_t^0})}\mu(f_t,y)=0$. Therefore $S(f_t|_{X_t^0})=\emptyset$ and hence $F$ is good.

On the other hand, if $F$ is good, there is a representative defined in $D\times
U$, where $D$, $U$ are open neighbourhoods of the origin in
$\mathbb{C}$, $\mathbb{C}^N$ respectively, such that
$X_t\setminus\{0\}$ is smooth and $f_t$ is regular on
$X_t\setminus\{0\}$, for any $t\in D$. Thus $S(X_t)=\{0\}$ and $S(f_t)=\{0\}$ on $U$. Therefore, by Theorem \ref{formula do numero de milnor}, $\mu(f)=\mu(f_t)$. \end{proof}

Another consequence of Theorem \ref{formula do numero de milnor} is the fact that the Milnor number of a function is upper semicontinuous:

\begin{cor} \label{uper} For all $x\in S(f_t)$ and for all $t$ small enough,
$$
\mu(f_t,x)\le \mu(f).
$$
\end{cor}

The next lemma is also an interesting consequence of Theorem \ref{formula do numero de milnor}.

\begin{lem} \label{consequencia do cor 4.4 bruna} If $(\mathcal{X},0)$ is a good and topologically trivial family, then $$\displaystyle\sum_{y\in S(f_t|_{X_t^0})}\nu(X_t,y)+\displaystyle\sum_{y\in S(f_t)}\nu(X_t\cap f_t^{-1}(0),y)=\nu(X\cap f^{-1}(0),0).$$ \end{lem}

\begin{proof} Since $(\mathcal{X},0)$ is good there exists a representative defined in $D\times U$, where $D$ and $U$ are open neighbourhoods of the origin in $\mathbb{C}$ and $\mathbb{C}^N$ respectively, such that $S(X_t)=\{0\}$ on $U$, for all $t\in D$. We consider $F$ defined in this representative of $(\mathcal{X},0)$. By Lê-Greuel's formula (\ref{Le-Greuel}) and Theorem \ref{formula do numero de milnor}, $$\displaystyle\sum_{y\in S(f_t|_{X_t^0})}\mu(f_t,y)+\mu(f_t)=\nu(X,0)+\nu(X\cap f^{-1}(0),0).$$

Applying Lê-Greuel's formula again to each $\mu(f_t)$, $$\displaystyle\sum_{y\in S(f_t|_{X_t^0})}\mu(f_t,y)+(\nu(X_t,0)+\nu(X_t\cap f_t^{-1}(0),0))=\nu(X,0)+\nu(X\cap f^{-1}(0),0).$$

It follows from \cite[Corollary 
4.4]{Bruna 3} that $\nu(X_t,0)=\nu(X,0)$. Therefore $$\displaystyle\sum_{y\in S(f_t|_{X_t^0})}\mu(f_t,y)+\nu(X_t\cap f_t^{-1}(0),0)=\nu(X\cap f^{-1}(0),0).$$

Again, by Lê-Greuel's formula applied to $\mu(f_t,y)$, $$\displaystyle\sum_{y\in S(f_t|_{X_t^0})}(\nu(X_t,y)+\nu(X_t\cap f_t^{-1}(0),y))+\nu(X_t\cap f_t^{-1}(0),0)=\nu(X\cap f^{-1}(0),0).$$

Hence $$\displaystyle\sum_{y\in S(f_t|_{X_t^0})}\nu(X_t,y)+\displaystyle\sum_{y\in S(f_t)}\nu(X_t\cap f_t^{-1}(0),y)=\nu(X\cap f^{-1}(0),0).$$ \end{proof}

\begin{cor} If $F$ is Whitney equisingular, then $F$ is $\mu$-constant and $m_i(X_t,0)$, $i=0,\ldots,d$, are constant on $t\in D$, where $d$ is the dimension of $X$ and $m_i(X_t,0)$ is the $i$th polar multiplicity of $(X_t,0)$. \end{cor}

\begin{proof} Since $F$ is Whitney equisingular, $(\mathcal{X},0)$ is Whitney equisingular too. Hence $m_i(X_t,0)$, $i=0,\ldots,d$, are constant on $t\in D$ (see \cite[Theorem 5.3]{Bruna 3}). 

Moreover, since $F$ is good, it is $\mu$-constant by Corollary \ref{teo final}. \end{proof}

In the particular case in which the IDS is an ICIS $(s=1)$, we present sufficient conditions for $F$ to be Whitney equisingular.

\begin{teo} If $s=1$, $F$ is $\mu$-constant, $m_i(X_t,0)$, $i=0,\ldots,d$, and $m_k(X_t\cap f_t^{-1}(0),0)$, $k=0,\ldots,d-1$, are constant on $t\in D$, then $F$ is Whitney equisingular.  
\end{teo}

\begin{proof} We have that $F$ is $\mu$-constant and $(\mathcal{X},0)$ is also good, since it is $\mu$-constant and is a family of ICIS. Hence, $F$ is good by Corollary \ref{teo final}. Therefore there exists a representative defined in $D\times U$, where $D$ and $U$ are open neighbourhoods of the origin in $\mathbb{C}$ and $\mathbb{C}^N$ respectively, such that $X_t\setminus \{0\}$ is smooth and $f_t$ is regular on $X_t\setminus \{0\}$, for any $t\in D$.

	This representative of $F$ admits a regular stratification given by $$\mathcal{A}=
	\{\mathcal{X}\setminus F^{-1}(\mathbb{C}\times\{0\}),F^{-1}(\mathbb{C}\times\{0\})\setminus (\mathbb{C}\times\{0\}),\mathbb{C}\times\{0\}\}$$ in the source and $$\mathcal{A'}=\{(\mathbb{C}\times\mathbb{C})\setminus(\mathbb{C}\times\{0\}),\mathbb{C}\times\{0\}\}$$ in the target. We set $A=\mathcal{X}\setminus F^{-1}(\mathbb{C}\times\{0\}), B=F^{-1}(\mathbb{C}\times\{0\})\setminus (\mathbb{C}\times\{0\}), C=\mathbb{C}\times\{0\}, A'=(\mathbb{C}\times\mathbb{C})\setminus(\mathbb{C}\times\{0\})$ and $B'=\mathbb{C}\times\{0\}$.
	
We remark that $F$ takes each stratum in $\mathcal{A}$ to a stratum of $\mathcal{A'}$ submersively. 
	
	Moreover, $A$ is an open set in $\mathcal{X}^0$, where $\mathcal{X}^0=\mathcal{X}\setminus (\mathbb{C}\times\{0\})$. Since $m_i(X_t,0)$, $i=0,\ldots,d$, are constant on $t\in D$, then $(\mathcal{X}^0,C)$ satisfies the Whitney conditions (see \cite[Theorem 1]{Gaffney definicao de multiplicidade}). Hence $(A,C)$ also satisfies the Whitney conditions. Furthermore, $m_k(X_t\cap f_t^{-1}(0),0)$, $k=0,\ldots,d-1$, are constant on $t\in D$, therefore $(B,C)$ satisfies the Whitney conditions.
	
We prove now that the pair $(A,B)$ satisfies the Whitney conditions. For each $(t,y)\in B$ there is a diffeomorphism $\varphi\colon U\to U'$, where $U$ and $U'$ are open neighbourhoods in $\mathbb{C}^{d+1}$ and $\mathcal{X}\setminus (\mathbb{C}\times\{0\})$ respectively with $(t,y)\in U$ and $d$ is dimension of $X$. The pair  $(\varphi^{-1}(A\cap U'),\varphi^{-1}(B\cap U'))$ satisfies the Whitney conditions because $\varphi^{-1}(A\cap U')$ is an open set in $\mathbb{C}^{d+1}$ and thus $\varphi^{-1}(A\cap U')$ contains all the secants. Therefore $\mathcal{A}$ satisfies the Whitney equisingularity conditions (see \cite[Lemma 2.2]{Mather}). 

It is easy to prove that $(A,C)$ and $(B,C)$ satisfies Thom's condition $A_F$ because $F|_{C}(t,0)=(t,0)$ and therefore $\ker(d(F|_{C})(t,0))=\{0\}$.
The pair $(A,B)$ satisfies Thom's condition $A_F$ because $F$ is a submersion at $\mathcal{X}^0$.                                                                                                                                                                                                                                                                                                                                                                                                                                                                                                                                                                                                                                                                                                                                                                                                                                                                                                                                                                                                                                                                                                                                                                 
 Hence $F$ is Whitney equisingular. \end{proof}

Our goal now is to study the previous theorem for a family of IDS. To do this, we introduce the $(d-1)$th polar multiplicity of the fiber $Y:=X\cap f^{-1}(0)$ of a function germ $f\colon(X,0)\rightarrow (\mathbb{C},0)$ with isolated singularity. Before, we need some results.

\begin{lem} \label{imitando o Teorema 3.4 Bruna} Assume that $(\mathcal{X},0)$ is a determinantal smoothing of $(X,0)$ and $F$ is such that $f_t$ is regular, for $t\neq 0$ sufficiently small. Then for all $t\neq 0$ sufficiently small, $$\nu(Y,0)=(-1)^{d-1}(\chi(Y_t)-1),$$ where $Y_t=X_t\cap f_t^{-1}(0)$ and $Y=X\cap f^{-1}(0)$.
	
\end{lem}

\begin{proof} We take a representative $\Psi\colon D\times U\to M_{m,n}$, where $D, U$ are small enough open balls centered at the origin in $\mathbb{C}, \mathbb{C}^N$ respectivelly and such that $Y_t$ is smooth and $\rank(\psi_t(x))=s-1$ for all $x\in Y_t$ and for all $t\in D\setminus\{0\}$. We take $W$ the nonempty Zariski open set given by \cite[Lemma 3.1]{Daiane}, that is, $$W=(\mathbb{C}^N\times M_{m,n}\times \mathbb{C})\setminus K,$$ where \begin{center} $K=\{\alpha=(b,A,e)\in \mathbb{C}^N\times M_{m,n}\times \mathbb{C} \, \, |\, \, Y_{\alpha} \, \, \mbox{is not regular or} \, \, \rank(\psi_A(x))<s-1, \, \, \mbox{for some} \, \, x\in Y_{\alpha}\}$. \end{center}

We construct a new deformation as the sum of two deformations $X_t$ and $X_A$, that is, for $A\in M_{m,n}$ and $t\in D$ we put $$X_{(A,t)}=(\psi_t+A)^{-1}(M_{m,n}^s).$$	

Moreover, we construct $$f_{(b,t)}|_{X_{(A,t)}}\colon X_{(A,t)}\to \mathbb{C}, \, \,  f_{(b,t)}(x_1,\ldots,x_N):=f_t(x_1,\ldots,x_N)+b_1x_1+\ldots+b_Nx_N.$$ 
For each $(\alpha,t)=(b,A,e,t)\in\mathbb{C}^N\times M_{m,n}\times \mathbb{C}\times D$  we define $Y_{(\alpha,t)}:=X_{(A,t)}\cap f_{(b,t)}^{-1}(e)$.

We show now that there is a nonempty Zariski open subset $W_0\subset \mathbb{C}^N\times M_{m,n}\times \mathbb{C}\times D$ such that
\begin{enumerate}
\item $Y_{(\alpha,t)}$ is smooth and $\rank(\psi_t(x)+A)=s-1$, for all $x\in Y_{(\alpha,t)}$ and for all $(\alpha,t)\in W_0$;

\item $\chi(Y_{(\alpha,t)})$ does not depend on $(\alpha,t)\in W_0$.
\end{enumerate}

\medskip
For the first part we consider \begin{center} $\tilde{C}=\{(\alpha,t,x) \, | \, x \, \,\mbox{is a singular point of} \, \, Y_{(\alpha,t)} \, \, \mbox{or} \, \, \rank(\psi_t(x)+A)<s-1\}$ \end{center} and $W_0=(\mathbb{C}^N\times M_{m,n}\times \mathbb{C}\times\mathbb{C})\setminus C$, where \begin{center} $C=\{(\alpha,t) \, | \, Y_{(\alpha,t)} \, \,\mbox{is not regular or} \, \, \rank(\psi_t(x)+A)<s-1 \, \,\mbox{for some} \, \, x\in Y_{(\alpha,t)}\}$. \end{center}
By the Jacobian criterion, $\tilde{C}$ is an analytic set (see \cite[Theorem 4.3.15]{Pfister}).

We consider $\pi_1\colon(\tilde{C},0)\rightarrow (\mathbb{C}^N\times M_{m,n}\times \mathbb{C}\times \mathbb{C},0)$, $\pi_1(\alpha,t,x):=(\alpha,t)$. It is easy to see that $\pi_1^{-1}(0)=\{0\}$ because $f$ has isolated singularity at $0$. Then, by \cite[Theorem 3.4.24]{Pfister}, $\pi_1$ is a finite map. Thus the image $C=\pi_1(\tilde{C})$ is analytic (see \cite[Theorem 2]{Grauert}).

We take $B\subset U$ a closed ball around the origin in $\mathbb{C}^N$, small enough, such that $0$ is the only singularity of $X$ in $B$ and $\partial B$ is transverse to $X$ and to $X\cap f^{-1}(0)$.

Let $$\begin{array}{cccc}
\xi\colon & \mathbb{C}^N\times M_{m,n}\times\mathbb{C}\times\mathbb{C}\times B & \to &  M_{m,n}\times\mathbb{C} \\
& (\alpha,t,x) & \mapsto & (\psi_t(x)+A,f_{(b,t)}(x)-e).
\end{array}$$ 
The maps $\xi$ and $\partial \xi=\xi|_{\mathbb{C}^N\times M_{m,n}\times\mathbb{C}\times\mathbb{C}\times \partial B}$ are submersions. Then, by the Transversality Theorem \cite{Guillemin Pollack} for allmost all $(\alpha,t)$, $\xi_{(\alpha,t)}$ and $\partial \xi_{(\alpha,t)}$ are transverse to $\Sigma^{s-i}\times \{0\}$, where $\xi_{(\alpha,t)}(x)=\partial\xi_{(\alpha,t)}(x)=\xi(\alpha,t,x)$. But $$\dim \mathbb{C}^N+\dim\Sigma^{s-i}\times \{0\}=N+mn-(m-s+i)(n-s+i)<mn$$ if $i>1$. Thus $\xi_{(\alpha,t)}(\mathbb{C}^N)\cap (\Sigma^{s-i}\times\{0\})=\emptyset$.
Hence $Y_{(\alpha,t)}=\xi_{(\alpha,t)}^{-1}(\Sigma^{s-1}\times\{0\})$ is smooth and $C$ is proper.

\medskip
We show now the second part. Let $$\begin{array}{cccc}
\pi\colon & \xi^{-1}(\Sigma^{s-1}\times \{0\}) & \to &  W_0 \\
& (\alpha,t,x) & \mapsto & (\alpha,t).
\end{array}$$ 
We use the Ehresmann Fibration Theorem on manifolds with border (\cite{Ehresmann fibration}) to show that $\pi$ is a fibration. For this we prove first that $\pi$ is proper. Indeed, let $K\subset W_0$ be compact. Then $K$ is closed and bounded. We note $$\pi^{-1}(K)\subseteq K\times B\subseteq \xi^{-1}(\Sigma^{N-1}\times \{0\}).$$
Since $\pi\colon\xi^{-1}(\Sigma^{s-1}\times \{0\})\to W_0$ is continuous then $\pi^{-1}(K)$ is closed in $\xi^{-1}(\Sigma^{s-1}\times \{0\})$. Thus $\pi^{-1}(K)$ is closed in $K\times B$. Therefore $\pi^{-1}(K)$ is also compact.

Now we show that $\pi$ and $\partial \pi$ (the restriction of $\pi$ to $\partial(\xi^{-1}(\Sigma^{s-1}\times \{0\})$) are both submersions. In fact, for all 
$(\alpha,t)\in W_0$, $\xi_{(\alpha,t)}$ and $\partial\xi_{(\alpha,t)}$ are both transverse to $\Sigma^{s-1}\times \{0\}$.
But then it follows from the proof of the Transversality Theorem \cite{Guillemin Pollack} that $(\alpha,t)$ is a regular value of both  $\pi$ and $\partial \pi$.

Therefore, $\pi$ is a fibration over the connected set $W_0$. We have \begin{eqnarray*}
	\pi^{-1}(\alpha,t)&=&\{(\alpha,t,x)\,\, | \, \, (\alpha,t,x)\in \xi^{-1}(\Sigma^{s-1}\times\{0\})\}\\
	&=&\{(\alpha,t,x) \,\, | \, \,  \xi_{(\alpha,t)}(x)\in \Sigma^{s-1}\times\{0\}\}\\
	&=&\{(\alpha,t)\}\times Y_{(\alpha,t)}. \end{eqnarray*}
Thus, 
$\chi(Y_{(\alpha,t)})$ does not depend on $(\alpha,t)$ in $W_0$.

We consider now $\alpha \in W$ and $t\in D\setminus\{0\}$. Then $(\alpha,0)\in W_0$ and $(0,0,0,t)\in W_0$ because $Y_{(\alpha,0)}=Y_{\alpha}$  is smooth by \cite[Lemma 3.1]{Daiane}, $\rank(\psi_0(x)+A)=\rank(\psi(x)+A)=s-1$, for all $x\in Y_{\alpha}$. Moreover, $X_{(0,t)}\cap f_{(0,t)}=X_t\cap f_t^{-1}(0)$ is smooth and $\rank(\psi_t(x))=s-1$, for all $x\in X_t\cap f_t^{-1}(0)$.

Hence, \begin{eqnarray*}
	\nu(Y,0)&=&(-1)^{d-1}(\chi(Y_{\alpha})-1)\\
	&=&(-1)^{d-1}(\chi(X_A\cap f_b^{-1}(e))-1)\\
	&=&(-1)^{d-1}(\chi(X_{(A,0)}\cap f_{(b,0)}^{-1}(e))-1)\\
	&=&(-1)^{d-1}(\chi(X_{(0,t)}\cap f_{(0,t)}^{-1}(0))-1)\\
	&=&(-1)^{d-1}(\chi(Y_{t})-1).
\end{eqnarray*} \end{proof}

\begin{lem} \label{lema 4.1 bruna para a fibra}  With the notation of the proof of Lemma \ref{imitando o Teorema 3.4 Bruna},
there is a linear function $p\colon\mathbb{C}^N\to \mathbb{C}$ such that the set $W\subset \mathbb{C}^N\times M_{m,n}\times \mathbb{C}\times \mathbb{C}$ of points $(\alpha,t)=(b,A,e,t)$ in which $p|_{Y_{(\alpha,t)}}$ is a Morse function is a nonempty Zariski open set.
\end{lem}

\begin{proof} We take $(\alpha_0,t_0)=(b_0,A_0,e_0,t_0)\in \mathbb{C}^N\times M_{m,n}\times \mathbb{C}\times \mathbb{C}$ such that $Y_{(\alpha_0,t_0)}$ is smooth (there is $(\alpha_0,t_0)$ by the first part of proof of Lemma \ref{imitando o Teorema 3.4 Bruna}). We consider $a=(a_1,\ldots,a_N)\in \mathbb{C}^N$ and we denote by $p_a\colon\mathbb{C}^N\to \mathbb{C}$ the linear function $p_a(x_1,\ldots,x_N)=a_1x_1+\ldots+a_Nx_N$. By \cite[Lemma A.2]{Bruna 2}, taking $f\equiv 0$, we can choose a point $a\in \mathbb{C}^N$ such that $p_a|_{Y_{(\alpha_0,t_0)}}$ and $p_a|_{Y\setminus \{0\}}$ are both Morse functions. We choose one of these $p_a$'s.
	
	We define $\tilde{C}$ as the subset of points $(\alpha,t,x)\in \mathbb{C}^N\times M_{m,n}\times \mathbb{C}\times \mathbb{C}\times \mathbb{C}^N$ such that either $x$ is a singular point of $Y_{(\alpha,t)}$ or $x$ is a degenerate critical point of $p_a|_{Y_{(\alpha,t)}}$.
	By Lemma \ref{conjunto analitico}, $\tilde{C}$ is analytic.

	Let $$\begin{array}{cccc}
	\pi\colon & (\tilde{C},0) & \to & (\mathbb{C}^N\times M_{m,n}\times \mathbb{C}\times \mathbb{C},0) \\
	& (\alpha,t,x) & \mapsto & (\alpha,t).
	\end{array}$$
	
	It is easy to see that $\pi^{-1}(0)=\{0\}$ because if $x$ is a singular point of $Y$ then $x=0$ because $f$ has isolated singularity at $0$ and if $x$ is a degenerate point of $Y$ then $x=0$ because $p_a$ was chosen such that $p_a|_{Y\setminus \{0\}}$ is a Morse function. Then, by \cite[Theorem 3.4.24]{Pfister}, $\pi$ is a finite map. Thus the image $C:=\pi(\tilde{C})$ is analytic (see \cite[Theorem 2]{Grauert}). Therefore, $W=(\mathbb{C}^N\times M_{m,n}\times \mathbb{C}\times \mathbb{C})\setminus C$ is a Zariski open set. Since $(\alpha_0,t_0)\in W$ it is nonempty. \end{proof}

\begin{lem} \label{imitando o Lema 4.2 Bruna}  Let $p\colon\mathbb{C}^N\to \mathbb{C}$ be a generic linear function. Then, \begin{eqnarray*}
		\nu(Y\cap p^{-1}(0),0)&=&(-1)^{d-2}(\chi(Y\cap p^{-1}(c))-1)\\
		&=&(-1)^{d-2}(\chi(Y_{\alpha}\cap p^{-1}(0))-1)\\
		&=&(-1)^{d-2}(\chi(Y_{\alpha}\cap p^{-1}(c))-1),
	\end{eqnarray*} where $d=\dim(X,0)$, $\alpha=(b,A,e)\in \mathbb{C}^N\times M_{m,n}\times \mathbb{C}$ is generic, $c\in\C$ is generic and $Y_{\alpha}=X_A\cap f_b^{-1}(e)$.
	
\end{lem}

\begin{proof} By taking an appropriate linear change of coordinates, we may assume  $p(x_1,\ldots,x_N)=x_N$. We write for each $c\in \mathbb{C}\setminus \{0\}$ $$\psi_c(x_1,\ldots,x_{N-1})=\psi(x_1,\ldots,x_{N-1},c).$$
	
We have the following identifications:
\begin{enumerate}

\item $(Y\cap p^{-1}(0),0)$ corresponds to $(Z\cap f^{-1}(0),0)$, where $(Z,0)=(\psi_0^{-1}(M_{m,n}^s),0)$;

\item $Y_{\alpha}\cap p^{-1}(0)$ corresponds to $Z_A\cap f_b^{-1}(e)$, where $Z_A=(\psi_0+A)^{-1}(M_{m,n}^s)$;

\item $Y\cap p^{-1}(c)$  is homeomorphic to $Z_c\cap f^{-1}(0)$, where $Z_c=(\psi_c)^{-1}(M_{m,n}^s)$;

\item $Y_{\alpha}\cap p^{-1}(c)$ is isomorphic to $Z_{(A,c)}\cap f_{(b,c)}^{-1}(e)$, where $Z_{(A,c)}=(\psi_c+A)^{-1}(M_{m,n}^s)$.

\end{enumerate}

Using the same arguments of the proof of Lemma \ref{imitando o Teorema 3.4 Bruna} we have the desired result. \end{proof}

Let $p\colon\mathbb{C}^N\to \mathbb{C}$ be a generic linear function. We have that $(X\cap p^{-1}(0),0)$ is also an IDS in the hyperplane $p^{-1}(0)$. 
Moreover $(X\cap p^{-1}(0))_A=X_A\cap p^{-1}(0)$. Indeed, with the notations of Lemma \ref{imitando o Lema 4.2 Bruna}, $$(X\cap p^{-1}(0))_A=(\psi_0+A)^{-1}(M_{m,n}^s)=(\psi+A)^{-1}(M_{m,n}^s)\cap p^{-1}(0)=X_A\cap p^{-1}(0).$$

Applying \cite[Definition 5.3]{Bruna 2} to the fiber of the IDS $(X\cap p^{-1}(0),0)$ we have $$\nu(X\cap p^{-1}(0)\cap f^{-1}(0),0):=(-1)^{d-2}(\chi(X_A\cap p^{-1}(0)\cap f_b^{-1}(e))-1),$$ with $(b,A,e)\in\mathbb{C}^N\times M_{m,n}\times
\mathbb{C}$ generic values such that $X_A\cap p^{-1}(0)$ is smooth,
$f_b|_{X_A\cap p^{-1}(0)}$ is a Morse function and $e$ is a regular value of
$f_b|_{X_A\cap p^{-1}(0)}$. 

\begin{teo}\label{imitando o Teorema 4.3 Bruna} Let $f\colon(X,0)\to(\C,0)$ be a function with isolated singularity on an IDS $(X,0)$.  Let $p\colon\mathbb{C}^N\to \mathbb{C}$ be a generic linear function, $\alpha=(b,A,e)\in \mathbb{C}^N\times M_{m,n}\times\mathbb{C}$ generic and $Y_{\alpha}=X_A\cap f_b^{-1}(e)$. Then, $$\sharp S(p|_{Y_{\alpha}})=\nu(Y,0)+\nu(Y\cap p^{-1}(0),0).$$  \end{teo}

\begin{proof} We choose $\alpha=(b,A,e)\in \mathbb{C}^N\times M_{m,n}\times\mathbb{C}$ generic such that $Y_{\alpha}$ is smooth and $p|_{Y_{\alpha}}$ is a Morse function (there is $\alpha$ by \cite[Lemma 3.1]{Daiane} and by the proof of Lemma \ref{lema 4.1 bruna para a fibra} for $t=0$). Let $c\in \mathbb{C}$ be a regular value of $p|_{Y_{\alpha}}$. By \cite[Theorem A.5]{Bruna 2 erratum} and Lemma \ref{imitando o Lema 4.2 Bruna}, \begin{eqnarray*}
		\sharp S(p|_{Y_{\alpha}})&=&(-1)^d(\chi(p^{-1}(c)\cap Y_{\alpha})-\chi(Y_{\alpha}))\\
		&=&(-1)^d(\chi(p^{-1}(c)\cap Y_{\alpha})-1+1-\chi(Y_{\alpha}))\\
		&=&(-1)^d((-1)^{d-2}\nu(Y\cap p^{-1}(0),0)+(-1)(-1)^{d-1}\nu(Y,0))\\
		&=&\nu(Y\cap p^{-1}(0),0)+\nu(Y,0).
	\end{eqnarray*} \end{proof}
	
We remark that $\sharp S(p|_{Y_{\alpha}})$ depends neither on the chosen linear function $p$ nor on $\alpha$. 
	Indeed, by Lê and Teissier \cite{Le Teissier} and Lemma \ref{imitando o Lema 4.2 Bruna} we have $$Eu(Y,0)=\chi(Y\cap p^{-1}(t))=(-1)^{d-2}\nu(Y\cap p^{-1}(0),0)+1,$$ where $Eu(Y,0)$ is the local Euler obstruction. 
By Theorem \ref{imitando o Teorema 4.3 Bruna} $$\sharp S(p|_{Y_{\alpha}})=\nu(Y,0)+(-1)^{d-2}Eu(Y,0)+(-1)^{d-1}.$$ 
Therefore, we can finally write the definition of the $(d-1)$th polar multiplicity of the fiber $(Y,0)$.

\begin{defi} With the notation of Theorem \ref{imitando o Teorema 4.3 Bruna}, we define the \emph{$(d-1)$th polar multiplicity} of $(Y,0)$ as
 $$m_{d-1}(Y,0):=\sharp S(p|_{Y_{\alpha}}),$$ where $p\colon\mathbb{C}^N\to \mathbb{C}$ is a generic linear function, $\alpha=(b,A,e)\in \mathbb{C}^N\times M_{m,n}\times\mathbb{C}$ is generic and $d$ is the dimension of $(X,0)$.
 \end{defi}

In the case of an ICIS of dimension $d$, Jorge Perez and Saia in \cite{Jorge Perez e Saia} prove that $$1+(-1)^d\mu(X,0)=\sum_{i=0}^d(-1)^im_i(X,0).$$
Since for $f\colon(X,0)\rightarrow (\mathbb{C},0)$ a holomorphic function germ with an isolated singularity we have $(X\cap f^{-1}(0),0)$ is an ICIS too, of dimension $d-1$, then we can apply the above formula to it.

In \cite{Bruna 2} it is proved a similar result for an IDS $(X,0)$ of dimension $d$, that is, $$Eu(X,0)+(-1)^{d}m_d(X,0)=1+(-1)^{d}\nu(X,0).$$
We have the following result that guarantees us that this formula is also valid for the fiber of a function germ with isolated singularity defined on an IDS. 

\begin{cor} \label{formula para a obstrucao de euler da fibra} Let $(X,0)$ be an IDS, $f\colon(X,0)\rightarrow (\mathbb{C},0)$ a holomorphic function germ with an isolated singularity and $Y=X\cap f^{-1}(0)$. Then, $$Eu(Y,0)+(-1)^{d-1}m_{d-1}(Y,0)=1+(-1)^{d-1}\nu(Y,0).$$ \end{cor}

With this we get for $(Y,0)$ a generalization of the following formula proved in \cite{Bruna 3} for an IDS $(X,0)$ $d$-dimensional $$\nu(X,0)=(-1)^{d}(\sum_{i=0}^{d}(-1)^{i}m_i(X,0)-1).$$

\begin{cor}\label{soma alternada das multiplicidades} Let $(X,0)$ be an IDS, $f\colon(X,0)\rightarrow (\mathbb{C},0)$ a holomorphic function germ with an isolated singularity and $Y=X\cap f^{-1}(0)$. Then, $$\nu(Y,0)=(-1)^{d-1}(\sum_{i=0}^{d-2}(-1)^{d-i-2}m_i(Y,0)-1)+m_{d-1}(Y,0).$$ \end{cor}

\begin{proof} By Corollary \ref{formula para a obstrucao de euler da fibra}, $$\nu(Y,0)=(-1)^{d-1}Eu(Y,0)+m_{d-1}(Y,0)+(-1)^{d}.$$
	
	By \cite{Le Teissier}, $$Eu(Y,0)=\sum_{i=0}^{d-2}(-1)^{d-i-2}m_i(Y,0).$$

	Therefore, $$\nu(Y,0)=(-1)^{d-1}(\sum_{i=0}^{d-2}(-1)^{d-i-2}m_i(Y,0)-1)+m_{d-1}(Y,0).$$ \end{proof}

Since $\nu(X,0)$ and $\nu(Y,0)$ can be written in terms of the multiplicities $m_i(X,0)$, $i=0,\ldots,d$, and $m_k(Y,0)$, $k=0,\ldots,d-1$, respectively, then the Lê-Greuel formula implies the following result.

\begin{cor} \label{multiplicidades constantes implicam F cte} Let $F=\{f_t\colon(X_t,0)\to (\C,0)\}_{t\in D}$ be a family of functions with isolated singularitiy on IDS. If $m_i(X_t,0)$, $i=0,\ldots,d$, and $m_k(Y_t,0)$, $k=0,\ldots,d-1$, are constant on $t\in D$, then $F$ is $\mu$-constant. \end{cor}
	
	Our goal now is to prove that if $m_k(Y_t,0)$ are constant, for all $k=1,\ldots,d-1$, then the pair $(F^{-1}(\mathbb{C}\times\{0\})\setminus(\mathbb{C}\times\{0\}),\mathbb{C}\times\{0\})$ satisfies the Whitney conditions. For this we need to extend some results of \cite{Bruna 3} for the fiber $(Y,0)$.

	\begin{lem}\label{formula para a m} Let $F\colon(\mathcal X,0)\to(\C\times\C,0)$ be any unfolding of a function $f\colon(X,0)\to(\C,0)$ with isolated singularity on an IDS $(X,0)$. For all $t$ small enough, $$m_{d-1}(Y,0)=\displaystyle\sum_{y\in S(p|_{Y_t^0})}\mu(p|_{Y_t},y)+\sum_{z\in S({Y_t})}m_{d-1}(Y_t,z),$$
	where $S(Y_t)$ is the singular locus of $Y_t$ and $Y_t^0=Y_t\setminus S(Y_t)$.
		
	\end{lem}
	
	\begin{proof} We choose a Milnor ball $B_{\epsilon}$ of $Y$ at the origin, so in $B_{\epsilon}$ the origin is the only singular point of $Y$.
		We denote by $z_1,\ldots,z_k$ and $y_1,\ldots,y_l$ the singular points of $Y_t$ and $p|_{Y_t^0}$, respectively.
		For each $i=1,\ldots,k$ and $j=1,\ldots,l$ we choose also Milnor discs $D_i$ for $Y_t,$ at $z_i$ and $E_j$  for $p|_{Y_t^0}$ at $y_j$ such that these discs are pairwise disjoint and are contained in $B_\epsilon$.

		Let $p_a$ be a generic linear deformation of $p$ such that $p_a|_{Y_{(\alpha,t)}}$ is a Morse function (there is such $p_a$ by the proof of Lemma \ref{lema 4.1 bruna para a fibra}). 
		Thus, \begin{eqnarray*}
			m_{d-1}(Y,0)&=&\sharp S(p_{a}|_{Y_{(\alpha,t)}})\\
			&=&\displaystyle\sum_{j=1}^l\sharp S(p_{a}|_{Y_{(\alpha,t)}\cap E_j})+\sum_{i=1}^k\sharp S(p_{a}|_{Y_{(\alpha,t)}\cap D_i})\\
			&=&\displaystyle\sum_{j=1}^l\mu(p|_{Y_t},y_j)+\sum_{i=1}^k m_{d-1}(Y_t,z_i).			
		\end{eqnarray*} \end{proof}
	
Let $(\mathcal{X},0)$ be a $(d+1)$-dimensional variety and $\pi\colon(\mathcal{X},0)\rightarrow (\mathbb{C},0)$. Following \cite{Teissier} we can consider the relative polar multiplicities, $m_i(\mathcal{X},\pi,0)$, defined as the multiplicity of the relative polar variety $\overline{\{(t,x) \, | \, x\in S(p|_{X_t^0})\}}$, where $p\colon\mathbb{C}^N\to\mathbb{C}^{d-i+1}$ is a generic linear projection and $i=0,\ldots,d$.

\begin{lem} \label{multiplicidade polar relativa}  If $(\mathcal{Y},0)=(\mathcal{X}\cap F^{-1}(\mathbb{C}\times \{0\}),0)$ is a good family and $m_{d-1}(Y_t,0)$ is constant, then $m_{d-1}(\mathcal{Y},\pi,0)=0$, where $\pi\colon(\mathcal{Y},0)\rightarrow (\mathbb{C},0)$ is the projection $\pi(t,x)=t$.
	
\end{lem}

\begin{proof} By Lemma \ref{formula para a m},  $$m_{d-1}(Y,0)=\displaystyle\sum_{y\in S(p|_{Y_t^0})}\mu(p|_{Y_t},y)+m_{d-1}(Y_t,0).$$ Since, $m_{d-1}(Y_t,0)$ is constant we get $\displaystyle\sum_{y\in S(p|_{Y_t^0})}\mu(p|_{Y_t},y)=0$. Hence $\mu(p|_{Y_t},y)=0$ for all $y\in S(p|_{Y_t^0})$ and we get $S(p|_{Y_t^0})=\emptyset$. We set $P(t,x)=(t,p(x))$ and we have $$m_{d-1}(\mathcal{Y},\pi,0)=m_0(\overline{S(P|_{\mathcal{Y}^0})},0)=m_0(\overline{\{(t,x) \, | \, x \in S(p|_{Y_t^0})\}},0)=m_0(\emptyset,0)=0.$$ \end{proof}
	 
	\begin{teo} \label{multiplicidade implica whitney} If $(\mathcal{Y},0)=(\mathcal{X}\cap F^{-1}(\mathbb{C}\times \{0\}),0)$ is a good family, then $m_k(Y_t,0)$, $k=0,\ldots,d-1$, are constant on $t\in D$ if and only if the pair $(\mathcal{Y}\setminus(\mathbb{C}\times\{0\}),\mathbb{C}\times\{0\})$ satisfies the Whitney conditions. \end{teo}
	
	\begin{proof} Assume that $m_k(Y_t,0)$, $k=0,\ldots,d-1$, are constant on $t\in D$. By Lemma \ref{multiplicidade polar relativa}, $m_{d-1}(\mathcal{Y},\pi,0)=0$. Thus by \cite[Corollary 5.12]{Gaffney 1} we get that $(\mathcal{Y}\setminus(\mathbb{C}\times\{0\}),\mathbb{C}\times\{0\})$ satisfies the Whitney conditions. 
		
	Conversely, assume that $(\mathcal{Y}\setminus(\mathbb{C}\times\{0\}),\mathbb{C}\times\{0\})$ satisfies the Whitney conditions. From the results of Teissier \cite{Teissier}, $m_k(\mathcal{Y},\pi,(t,0))$ is constant on $t\in D$ for $k=1,\ldots,d-1$. In particular, $m_{d-1}(\mathcal{Y},\pi,(t,0))$ is constant on $t\in D$. 
	
	Suppose that $m_{d-1}(\mathcal{Y},\pi,(t,0))=m_{d-1}(\mathcal{Y},\pi,0)\neq 0$ for $t\in D$.
	By assumption, $Y_t^0=Y_t\setminus\{0\}$, so  $(t,0)\notin S(P|_{\mathcal{Y}^0})$, where $P(t,x)=(t,p(x))$. We deduce that 
	$$
	T=D\times\{0\}\subset \overline{S(P|_{\mathcal{Y}^0})}\setminus S(P|_{\mathcal{Y}^0}),
	$$ 
	which is not possible because the set on the right hand side is $0$-dimensional. Hence, $m_{d-1}(\mathcal{Y},\pi,(t,0))=m_{d-1}(\mathcal{Y},\pi,0)= 0$ for $t\in D$.
	 Therefore $$m_{d-1}(\mathcal{Y},\pi,0)=m_0(\overline{S(P|_{\mathcal{Y}^0})},0)=m_0(\overline{\{(t,x) \, | \, x \in S(p|_{Y_t^0})\}},0)=0.$$
	 In other words, $S(P|_{\mathcal{Y}^0})=\emptyset$ and thus, $S(p|_{Y_t^0})=\emptyset$. Therefore $\displaystyle\sum_{y\in S(p|_{Y_t^0})}\mu(p|_{Y_t},y)=0$. By Lemma \ref{formula para a m} we conclude that $$m_{d-1}(Y,0)=m_{d-1}(Y_t,0).$$ Hence $m_{d-1}(Y_t,0)$ is constant.
	 For $k=0,\ldots,d-2$, $m_k(Y_t,0)$ are constant by \cite[Theorem 5.6]{Gaffney 1}.	\end{proof}

\begin{teo}\label{main} The family $F$ is Whitney equisingular if and only if $(\mathcal{X},0)$ is a good family, $m_i(X_t,0)$, $i=0,\ldots,d$, and $m_k(Y_t,0)$, $k=0,\ldots,d-1$, are constant on $t\in D$. \end{teo}

\begin{proof} Assume, first, that the family $F$ is Whitney equisingular. So $F$ is good, that is, there is a representative defined in $D\times U$, where $D$ and $U$ are open neighbourhoods of the origin in $\mathbb{C}$ and $\mathbb{C}^N$ respectively, such that $X_t\setminus \{0\}$ is smooth and $f_t$ is regular on $X_t\setminus \{0\}$, for any $t\in D$. Therefore the families $(\mathcal{X},0)$ and $(\mathcal{Y},0)$ are good too, where $\mathcal{Y}=\mathcal{X}\cap F^{-1}(\mathbb{C}\times \{0\})$.

	This representative of $F$ admits a regular stratification given by $$\mathcal{A}=
	\{\mathcal{X}\setminus \mathcal{Y},\mathcal{Y}\setminus (D\times\{0\}),D\times\{0\}\}$$ in the source. We set $A=\mathcal{X}\setminus \mathcal{Y}, B=\mathcal{Y}\setminus (D\times\{0\}), C=D\times\{0\}$.
	
	Since $\mathcal{A}$ satisfies the Whitney conditions then so does $(B,C)$. Thus, by Theorem \ref{multiplicidade implica whitney}, $m_k(Y_t,0)$, $k=0,\ldots,d-1$, are constant on $t\in D$.
	
	Moreover the pair $(\mathcal{X}\setminus C,C)$ satisfies the Whitney conditions. In fact, let $\{(t_i,x_i)\}$ be a sequence of points of $\mathcal{X}\setminus C$ and $\{(s_i,0)\}$ a sequence of points of $C$ such that $\{(t_i,x_i)\}\rightarrow (t,0)$, $\{(s_i,0)\}\rightarrow (t,0)$, $T_{(t_i,x_i)}(\mathcal{X}\setminus C)\rightarrow K$ and $sec((t_i,x_i),(s_i,0))\rightarrow L$, where $(t,0)\in C$.
	
	Since $\mathcal{X}\setminus C=(\mathcal{X}\setminus \mathcal{Y})\cup(\mathcal{Y}\setminus C)$, then there is a subsequence $\{(t_{i_k},x_{i_k})\}$ contained in $\mathcal{X}\setminus \mathcal{Y}$ or in $\mathcal{Y}\setminus C$.
	
	In the first case, $T_{(t_{i_k},x_{i_k})}(\mathcal{X}\setminus C)=T_{(t_{i_k},x_{i_k})}(\mathcal{X}\setminus \mathcal{Y})$ because $\mathcal{X}\setminus \mathcal{Y}$ is an open set in $\mathcal{X}\setminus C$. Since $(\mathcal{X}\setminus \mathcal{Y},C)$ satisfies the Whitney conditions, then $L\subseteq K$.
	
	Now, if $\{(t_{i_k},x_{i_k})\}$ is in $\mathcal{Y}\setminus C$, after taking again a subsequence if necessary, we can assume that $T_{(t_{i_k},x_{i_k})}(\mathcal{Y}\setminus C)$ has a limit $M$. Since $\mathcal{Y}\setminus C$ is a submanifold of $\mathcal{X}\setminus C$, we have $T_{(t_{i_k},x_{i_k})}(\mathcal{Y}\setminus C)$ is contained in $T_{(t_{i_k},x_{i_k})}(\mathcal{X}\setminus C)$, for all $k$. Considering the limits of these tangents spaces, we conclude that $M$ is contained in $K$. Moreover, the pair $(\mathcal{Y}\setminus C,C)$ satisfies the Whitney conditions and thus, $L\subseteq M\subseteq K$.

	Conversely, if all the polar multiplicities are constant, then $F$ is $\mu$-constant by Corollary \ref{multiplicidades constantes implicam F cte}. Moreover, $(\mathcal{X},0)$ is good. So $F$ is good by Corollary \ref{teo final}. Therefore there is a representative defined in $D\times U$, where $D$ and $U$ are open neighbourhoods of the origin in $\mathbb{C}$ and $\mathbb{C}^N$ respectively, such that $X_t\setminus \{0\}$ is smooth and $f_t$ is regular on $X_t\setminus \{0\}$, for any $t\in D$.

	This representative of $F$ admits a stratification given by $$\mathcal{A}=
	\{\mathcal{X}\setminus \mathcal{Y},\mathcal{Y}\setminus (D\times\{0\}),D\times\{0\}\}$$ in the source and $$\mathcal{A'}=\{(\mathbb{C}\times\mathbb{C})\setminus(\mathbb{C}\times\{0\}),\mathbb{C}\times\{0\}\}$$ in the target, where $\mathcal{Y}=\mathcal{X}\cap F^{-1}(\mathbb{C}\times \{0\})$. We set $A=\mathcal{X}\setminus \mathcal{Y}, B=\mathcal{Y}\setminus (D\times\{0\}), C=D\times\{0\}, A'=(\mathbb{C}\times\mathbb{C})\setminus(\mathbb{C}\times\{0\})$ and $B'=\mathbb{C}\times\{0\}$.
	
	We remark that $F$ takes each stratum in $\mathcal{A}$ to a stratum of $\mathcal{A'}$ submersively. 
Moreover $A$ is an open set in $\mathcal{X}^0$, where $\mathcal{X}^0=\mathcal{X}\setminus (\mathbb{C}\times\{0\})$. Since $(\mathcal{X},0)$ is good and $m_i(X_t,0)$, $i=0,\ldots,d$, are constant on $t\in D$, then $(\mathcal{X}^0,C)$ satisfies the Whitney conditions (see \cite[Theorem 5.3]{Bruna 3}). Hence $(A,C)$ also satisfies the Whitney conditions. 
	
Since $F$ is good then so is $(\mathcal{Y},0)$ and $m_k(Y_t,0)$, $k=0,\ldots,d-1$, are constant on $t\in D$. Hence, the pair $(B,C)$ satisfies the Whitney conditions by Theorem \ref{multiplicidade implica whitney}.
	
	We prove now that the pair $(A,B)$ satisfies the Whitney conditions. For each $(t,y)\in B$ there is a diffeomorphism $\varphi\colon U\to U'$, where $U$ and $U'$ are open neighbourhoods in $\mathbb{C}^{d+1}$ and $\mathcal{X}\setminus (\mathbb{C}\times\{0\})$ respectively with $(t,y)\in U$ and $d$ is dimension of $X$. The pair  $(\varphi^{-1}(A\cap U'),\varphi^{-1}(B\cap U'))$ satisfies the Whitney conditions because $\varphi^{-1}(A\cap U')$ is an open set in $\mathbb{C}^{d+1}$ and thus $\varphi^{-1}(A\cap U')$ contains all the secants. Therefore, $(A,B)$ satisfies the Whitney equisingularity conditions (see \cite[Lemma 2.2]{Mather}).

	It is easy to prove that $(A,C)$ and $(B,C)$ satisfy Thom's condition $A_F$ because $F|_{C}(t,0)=(t,0)$ and therefore $\ker(d(F|_{C})(t,0))=\{0\}$.
	And the pair $(A,B)$ satisfies Thom's condition $A_F$ because $F$ is a submersion on $\mathcal{X}^0$.                                                                                                                                                                                                                                                                                                                                                                                                                                                                                                                                                                                                                                                                                                                                                                                                                                                                                                                                                                                                                                                                                                                                                                 
	Hence $F$ is Whitney equisingular. \end{proof}

By \cite[Definition 1.5]{Teissier} we have $$\nu^*(X_t,0):=(\nu_0(X_t,0),\ldots,\nu_d(X_t,0))$$ where $\nu_j(X_t,0)=\nu(X\cap H^{N-(d-j)},0)$ with $\nu(X\cap H^{N-(d-j)},0)$ is the vanishing Euler characteristic of the IDS $X\cap H^{N-(d-j)}$ in the generic hyperplane $H^{N-(d-j)}$.

By \cite[Remark 1.6]{Teissier}, $\nu_0(X_t,0)=m_0(X_t,0)-1$. Furthermore, for $i>0$ $m_i(X_t,0)=\nu_i(X_t,0)+\nu_{i-1}(X_t,0)$. Thus, $m_i(X_t,0)$ are constant if and only if $\nu^*(X_t,0)$ is constant.

\appendix

\section{}

Let $W$ be an ideal in $\mathcal{O}_N$ generated by $g_1,\ldots,g_r$ and $h=(h_1,\ldots,h_p)\colon(\mathbb{C}^N,0)\rightarrow \mathbb{C}^p$ a map germ. For each $m=1,\ldots,N$ Nuño-Ballesteros and Saia \cite{Juanjo e Saia} define the Jacobian extension of $\rank m$ of $(h,W)$ as $$\Delta_m(h,W)=W+W',$$ where $W'$ is the ideal generated by the minors of order $m$ of the Jacobian matrix of $(h_1,\ldots,h_p,g_1,\ldots,g_r)$. Following this, we define inductively the iterated Jacobian extension of $h$ by $$J_i(h,W)=\left\{
\begin{array}{llllllll}
\Delta_{N-i_1+1}(h,W), {\text{ if $k=1$,}}\\
\Delta_{N-i_k+1}(h,J_{i_1,\ldots,i_{k-1}}(h,W)), \, \, \mbox{if} \, \, $k$ \, \, \mbox{is bigger than} \, \, $1$,
\end{array} 
\right.$$ where $i=(i_1,\ldots,i_k)$ is a Boardman symbol (i.e. $N\geq i_1\geq\ldots\geq i_k\geq 0$). By \cite[Lemma 2.2]{Juanjo e Saia} $J_i(h,W)$ does not depend on the generators of $W$.

\begin{lem}\label{conjunto analitico} Let $Y=\varphi^{-1}(0)\subset \mathbb{C}^N$ be an analytic variety of dimension $d$ where $\varphi\colon\mathbb{C}^N\to \mathbb{C}^p$ is a holomorphic map germ. Let $g\colon Y\to \mathbb{C}$ be a holomorphic function germ. The set of points $y\in Y$ such that either $y$ is a singular point of $Y$ or $y$ is a degenerate critical point of $g$ is $$\tilde{C}=v(J_{d,1}(g,\langle\varphi\rangle))\cup S(Y),$$
where $\langle\varphi\rangle$ is the ideal generated by the components of $\varphi$ and $S(Y)$ is the singular locus of $Y$.

\end{lem}

\begin{proof} 

We first assume that $0$ is a regular point of $Y$. By the Jacobian criterion, 
$\varphi$ has rank $p-d$ at $0$, hence there are $\varphi_{i_1},\ldots,\varphi_{i_{p-d}}$ such that $\hat{\varphi}=(\varphi_{i_1},\ldots,\varphi_{i_{p-d}})$ is a submersion. Moreover, $\langle\varphi\rangle=\langle\hat{\varphi}\rangle$ and $\hat{\varphi}^{-1}(0)=Y$ in a neighborhood of the origin. Let 
$$\begin{array}{cccc} \pi\colon & \Gamma(g) & \to & \mathbb{C} \\
	& (x,g(x)) & \mapsto & g(x),
	\end{array}$$ where $\Gamma(g)=\{(x,g(x)) \, \, | \, \, x\in Y\}$ is the graph of $g$.  
	
	We know that $\Gamma(g)=(\varphi')^{-1}(0)$, where $$\begin{array}{cccc} \varphi'\colon & \mathbb{C}^N\times \mathbb{C} & \to & \mathbb{C}^{p-d+1} \\
	& (x,s) & \mapsto & (\hat{\varphi}(x),g(x)-s).
	\end{array}$$

	By \cite[Theorem 5.1]{Juanjo e Saia}, the set of degenerate critical points of $\pi$ in $\Gamma(g)$ is $$
v(J_{d,1}(\varphi';x))=v(\langle {\varphi}' \rangle+I_{N-d+1}(J({\varphi}';x))+I_{N}(J({\varphi}',h';x))),
$$ where $\langle h_1',\ldots,h_l'\rangle=I_{N-d+1}(J(\varphi';x))$ and $h'=(h_1',\ldots,h_l')$. 
The notation $J_i(\cdot;x)$ means that we construct the Jacobian ideals by taking only partial derivatives with respect to the variables $x_1,\dots,x_N$.

	Since $$\begin{array}{cccc} \Gamma\colon & Y & \to & \Gamma(g) \\
	& x & \mapsto & (x,g(x))
	\end{array}$$ is a diffeomorphism, the degenerate critical points of $g$ are the inverse images of the degenerate critical points of $\pi$ by $\Gamma$, which is equal to the analytic set $$v(\langle {\hat{\varphi}} \rangle+I_{N-d+1}(J({\hat{\varphi}},g))+I_{N}(J({\hat{\varphi}},g,h))),$$ 
	where $\langle h_1,\ldots,h_l\rangle=I_{N-d+1}(J(\hat{\varphi},g))$ and $h=(h_1,\ldots,h_l)$. 
	But according to our previous definition, this is equal to the analytic set $v(J_{d,1}(g,\langle\hat{\varphi}\rangle))=v(J_{d,1}(g,\langle\varphi\rangle)).$

The case where $0$ is not regular  follows easily from the regular one, just by adding the singular locus $S(Y)$. \end{proof}

\bibliographystyle{amsplain}

\end{document}